\documentclass[11pt, amsthm]{amsart}

\usepackage[active]{srcltx}
\usepackage{calc,amssymb,amsthm,amsmath, ulem}
\usepackage{alltt}
\usepackage[left=1.35in,top=1.2in,right=1.35in,bottom=1.2in]{geometry}
\RequirePackage[dvipsnames,usenames]{color}

\theoremstyle{theorem}
\newtheorem{thm}{Theorem}[section]
\newtheorem{lemma}[thm]{Lemma}
\newtheorem{theorem}[thm]{Theorem}
\newtheorem{proposition}[thm]{Proposition}
\newtheorem{corollary}[thm]{Corollary}

\theoremstyle{definition}
\newtheorem{definition}[thm]{Definition}
\newtheorem{example}[thm]{Example}

\newtheorem{convention}[thm]{Convention}

\theoremstyle{remark}
\newtheorem{remark}[thm]{Remark}

\newtheorem{question}[thm]{Question}

\DeclareMathOperator{\Spec}{{Spec}}
\DeclareMathOperator{\Supp}{{Supp}}
\DeclareMathOperator{\Sing}{{Sing}}
\DeclareMathOperator{\Hom}{Hom}
\newcommand{\ba}{\mathfrak{a}}

\newcommand{\bP}{\mathbb{P}}
\DeclareMathOperator{\Ann}{{Ann}}
\DeclareMathOperator{\Image}{Image}

\newcommand{\cf}{{\itshape cf.} }

\DeclareMathOperator{\Div}{div}

\newcommand{\bm}{\mathfrak{m}}

\DeclareMathOperator{\HH}{H} 

\usepackage{hyperref}

\usepackage{enumerate}

\usepackage[all,cmtip]{xy}
\usepackage{verbatim}

\renewcommand{\O}{\mathcal O}

\def\MatrixDiagram#1{
$\begin{xy}
(0,12)*=0{}="1", (12,12)*=0{}="2",
(0,9)*=0{}="3", (12,9)*=0{}="4",
(0,6)*=0{}="5", (12,6)*=0{}="6",
(0,3)*=0{}="7", (12,3)*=0{}="8",
(0,0)*=0{}="9", (12,0)*=0{}="10",
(3,12)*=0{}="11", (3,0)*=0{}="12",
(6,12)*=0{}="13", (6,0)*=0{}="14",
(9,12)*=0{}="15", (9,0)*=0{}="16",
(3,9)*=0{}="a1", (3,6)*=0{}="a2", (3,3)*=0{}="a3",
(6,6)*=0{}="b1", (6,3)*=0{}="b2",
(9,3)*=0{}="c1",
"1";"2"  **@@{.} ,"3";"4"  **@@{.} ,"5";"6"  **@@{.} ,"7";"8"  **@@{.} ,"9";"10"  **@@{.} ,
"1";"9"  **@@{.} ,"11";"12"  **@@{.} ,"13";"14"  **@@{.} ,"15";"16"  **@@{.} ,"2";"10"  **@@{.}
@+"9"\PATH ~={**\dir{-}}
#1
\end{xy}
$
}

\begin{document}

 \title{An algorithm for computing compatibly Frobenius split subvarieties}
\author{Mordechai Katzman and Karl Schwede} 
\subjclass[2000]{14B05, 13A35}
\keywords{Frobenius map, Frobenius splitting, compatibly split, test ideal, algorithm, prime characteristic}
\address{Department of Pure Mathematics\\ University of Sheffield\\ Hicks Building, Hounsfield Road\\ Sheffield S3 7RH, United Kingdom}
\email{m.katzman@sheffield.ac.uk}
\address{Department of Mathematics\\ The Pennsylvania State University\\ University Park, PA, 16802, USA}
\email{schwede@math.psu.edu}
\thanks{The first author was partially supported by Royal Society grant TG091897 and EPSRC grant EP/H040684/1.}
\thanks{The second author was partially supported by a National Science Foundation postdoctoral fellowship and NSF DMS 1064485/0969145.}
\maketitle

\begin{abstract}
This paper describes an algorithm which produces all ideals compatible with a given surjective Frobenius near-splitting.
\end{abstract}




\section{Introduction}

Suppose that $X$ is an algebraic variety over a field of characteristic $p > 0$. If $X$ is Frobenius split with a Frobenius splitting $\phi : F^e_* \O_X \to \O_X$, then $X$ satisfies numerous remarkable properties.  The following varieties possess a Frobenius splitting: toric varieties, Schubert varieties, ordinary Abelian varieties and, at least when reduced to characteristic $p \gg 0$, Fano varieties.  In this context, it is very natural to study the compatibly $\phi$-split subvarieties ($Z \subseteq X$ whose ideal sheaf $I_Z$ satisfies $\phi(F^e_* I_Z) \subseteq I_Z$).
These special subvarieties play a fundamental role whenever Frobenius split varieties are studied (see \cite{BrionKumarFrobeniusSplitting}).  Recent independent work in \cite{KumarMehtaFiniteness}, and independently by the second author, \cite{SchwedeFAdjunction}, have shown that there are only finitely many such subvarieties, also see \cite{SharpGradedAnnihilatorsOfModulesOverTheFrobeniusSkewPolynomialRing} and \cite{EnescuHochsterTheFrobeniusStructureOfLocalCohomology}.  In this paper, building on the ideas from \cite{KumarMehtaFiniteness, SchwedeFAdjunction}, as well as ideas coming from tight closure theory, we exhibit an algorithm which computes all the compatibly $\phi$-split subvarieties.

While Frobenius split varieties need not be affine, in this paper we restrict ourselves to affine varieties.  This is not a terribly restrictive hypothesis since compatibly split subvarieties of a projective variety $X$ can be studied either on affine charts or by considering the affine cone over $X$.  Our main result is as follows:
\vskip 6pt
\begin{quote}
Given a ring $R$ and a surjective map $\phi : F^e_* R \to R$ (for example a Frobenius splitting), we exhibit an algorithm which produces all the $\phi$-compatible ideals.
\end{quote}
\vskip 6pt
At each step, the algorithm produces the unique smallest non-zero $\phi$-compatible ideal, the so-called test ideal.

Finally, we also explore a variant of this algorithm under the hypothesis that $\phi$ is not necessarily a Frobenius splitting (or even surjective).
This algorithm, and the original, have been implemented in Macaulay2 \cite{KatzmanSchwedeM2FSplitting}. 

\medskip
\noindent{\it Acknowledgements: }  The authors would like to thank both referees, Allen Knutson, Lance Miller and Kevin Tucker for numerous extremely useful comments on this paper.  The authors would also like to thank Jen-Chieh Hsiao for several discussions.
The authors first worked out the details of this algorithm in March 2010, when the second author visited the first funded by the Mathematics and Statistics Research Centre at the University of Sheffield.  This algorithm also has similarities to an algorithm due to Knutson-Lam-Speyer and implemented by Jenna Rajchgot, see \cite[Theorem 5.3]{KnutsonLamSpeyerProjectionsOfRichardson}.  The primary difference between these algorithms is the way singularities are handled.  Macaulay2  \cite{M2}, has been used extensively in the writing of this paper both in constructing and exploring examples, as well as implementing the algorithm described herein.

\section{Notation and background}

\begin{convention}
\label{conv.PolynomialConvention}
Through this paper all rings are commutative and of finite type over a perfect field $k$ of characteristic $p > 0$, or they are a localization of such a ring.
\end{convention}
The algorithm in this paper indeed also holds on a much larger class of rings: at a basic level, this algorithm works for \emph{any} ring of equal characteristic $p$ such that the Frobenius map is a finite map (this condition is often called being \emph{$F$-finite}).
We use the notation $F_* R$ (respectively $F^e_* R$) to denote $R$ viewed as an $R$-module via Frobenius (respectively, via $e$-iterated Frobenius).  More generally, for any $R$-module $M$, $F^e_* M$ denotes the module $M$ with the induced $R$-module structure via Frobenius.  Additionally, given an element $r \in R$, we will use $F^e_* r$ to denote the corresponding element of $F^e_* R$.  Finally, if $I = \langle f_1, \dots, f_m \rangle \subseteq R$ is an ideal, we use $I^{[p^e]}$ to denote the ideal $\langle f_1^{p^e}, \dots, f_n^{p^e} \rangle$. This is easily seen to be independent of the choice of generators of $I$.

\begin{definition}
We say that an $R$-linear map $\phi : F^e_* R \to R$ is a \emph{splitting of ($e$-iterated) Frobenius}, or simply a \emph{$F$-splitting}, if $\phi$ sends $F^e_* 1$ to $1$.  If $R$ has a Frobenius splitting, then we say that $R$ is \emph{$F$-split (by $\phi$)}.
\end{definition}

\begin{definition}
Given any $R$-linear map $\phi: F^e_* R \to R$ (not necessarily a splitting), we say that an ideal $J \subseteq R$ is \emph{$\phi$-compatible} if $\phi(F^e_* J) \subseteq J$, or simply that \emph{$J$ is compatible with $\phi$}.  If the $\phi$ is clear, sometimes we will only say that $J$ is \emph{compatible}.  If $\phi$ is indeed a Frobenius splitting, then we say that \emph{$J$ is compatibly ($\phi$-)split}.
\end{definition}

Given $\phi$ which is compatible with $J$ as above, then there always exists a commutative diagram:
\begin{equation}
\label{eq.CompatibleDiagram}
\xymatrix{
F^e_* R \ar[d] \ar[r]^{\phi} & R \ar[d]\\
F^e_* (R/J) \ar[r]_-{\phi/J} & R/J
}
\end{equation}
where the vertical arrows are the canonical surjections.  We will use $\phi/J$ to denote the induced map $F^e_* (R/J) \to R/J$ as pictured above.

The following well-known Lemma, which we will rely on heavily, follows immediately from the diagram above.

\begin{lemma}
\label{lem.CompatibleIdealsOfDiagrams}
 Assuming a commutative diagram (\ref{eq.CompatibleDiagram}) as above, the $\phi$-compatible ideals containing $J$ are in bijective correspondence with the $\phi/J$-compatible ideals of $R/J$.
\end{lemma}

The following theorem motivates the main question of this paper.  Its proof motivates the method of the algorithm

\begin{theorem}\cite{KumarMehtaFiniteness}, \cite{SchwedeFAdjunction}
If $\phi : F^e_* R \to R$ as above is surjective, then there are finitely many $\phi$-compatible ideals.
\end{theorem}
\begin{proof}
 Because the proof is motivating, we give a rough sketch of it here.  Reduce to the case that $R$ is a domain.  Fix a divisor $D_{\phi}$ on the normal locus of $X = \Spec R$ which corresponds to $\phi$ (in the usual sense of Frobenius splittings, see Proposition \ref{propFSplittingsProperties}(b) and \cite{BrionKumarFrobeniusSplitting}).  Then by Proposition \ref{propFSplittingsProperties}(b) below, all of the $\phi$-compatible ideals have support contained within
\[
 \Supp(D) \cup \Sing(X)
\]
where $\Sing(X)$ is the non-regular locus of $X$.
A roughly equivalent statement from the tight-closure perspective is as follows:  all of the non-zero $\phi$-compatible ideals contain the big test ideal $\tau_b(R, \phi)$ (the unique smallest non-zero $\phi$-compatible ideal).

Regardless, take $X_1$ to be the union of the $\phi$-compatible subvarieties (vanishing loci of the $\phi$-compatible ideals).  Either of the previous observations imply that the closure of $X_1$ is a proper closed subset of $X = \Spec R$.  Then repeat the process replacing $X = \Spec R$ by $\overline{X_1}$ and apply Noetherian induction.
\end{proof}

Therefore, in order to turn this theorem into an algorithm, one needs a way to identify the union of all subvarieties compatible with a given splitting $\phi$.  Equivalently and more algebraically, one needs to identify the smallest $\phi$-compatible ideal (that ideal is called a ``test ideal'').

We list some basic properties of Frobenius splittings that we will need in what follows.

\begin{proposition}
\label{propFSplittingsProperties}
Suppose that $R$ is as in Convention \ref{conv.PolynomialConvention} and that $X = \Spec R$.  We will assume that $\phi \in \Hom_R(F^e_* R, R)$ is a non-zero element.  Then:
\begin{itemize}
\item[(a)]  We have an isomorphism $\Hom_R(F^e_* R, R) \cong F^e_* \omega_R^{(1-p^e)} \cong F^e_* \Gamma(X, \O_X( (1-p^e)K_X) )$ of $F^e_* R$-modules.
In particular, if $R$ is Gorenstein then $\Hom_R(F^e_* R, R)$ is a locally free rank-one $F^e_* R$-module.
If $\Hom_R(F^e_* R, R)$ is free as an $F^e_* R$-module, we label a generator of this $F^e_*R$-module by $\Phi_R$.
\item[(b)]  By (a), $\phi$ corresponds to an effective divisor $D_{\phi}$ linearly equivalent to $(1-p^e)K_X$.  Furthermore, for every $\phi$-compatible ideal $J$, $V(J) \subseteq X$ is contained within the set $(\Supp D_{\phi}) \cup (\Sing X)$.
\item[(c)]  If $\phi$ is surjective, then the set of $\phi$-compatible ideals is a finite set of radical ideals closed under sum and primary decomposition.
\item[(d)]  Let $\phi^t$ denote the composition
$$F^{te}_* R  \xrightarrow[]{F^{(t-1)}_* \phi} F^{(t-1) e}_* R  \xrightarrow[]{F^{(t-2)}_* \phi} \dots \xrightarrow[]{F_* \phi} F^e_* R \xrightarrow[]{\phi} R .$$
Then if $J$ is $\phi$-compatible, it is also $\phi^t$-compatible for any $t > 0$.  If $\phi$ is surjective, then every $\phi^t$-compatible ideal is also $\phi$-compatible.
\item[(e)]  If $\phi$ is surjective, $R$ is a domain, and $\ba \neq 0$ is any ideal that vanishes on the set $(\Supp D_{\phi}) \cup (\Sing X)$, then
\[
\phi( F^e_* \ba ) \subseteq \phi^2 (F^{2e}_* \ba) \subseteq \phi^3(F^{3e}_* \ba) \subseteq \dots
\]
stabilizes to be the unique smallest $\phi$-compatible ideal.
\end{itemize}
\end{proposition}
\begin{proof}
Parts (a) and the first part of (b) are in \cite[Chapter 1]{BrionKumarFrobeniusSplitting} among many other places.
The last part of (b) can be found in \cite{KumarMehtaFiniteness} but also is an immediately consequence of the
theory of (sharp) test elements for pairs, see for example \cite{SchwedeFAdjunction}.   The finiteness of part (c), as mentioned before was independently obtained in \cite{KumarMehtaFiniteness} and \cite{SchwedeFAdjunction}.  The other parts of (c) are the same as for compatibly split ideals and can be found in \cite[Chapter 1]{BrionKumarFrobeniusSplitting}.

The first statement in (d) is obvious.  The second statement can be found in \cite[Proposition 4.1]{SchwedeCentersOfFPurity} although we also include a short proof for the convenience of the reader.  To this end, suppose that $\phi^t(F^{te}_* J) \subseteq J$ for some surjective $\phi : F^e_* R \to R$.  It follows that $\phi^t$ is also surjective from which it follows that $J$ is radical.  All of the associated primes of $J$ are also $\phi^t$-compatible and since an intersection of $\phi$-compatible ideals is again compatible, it is harmless to assume that $J$ is prime.  Finally, a prime ideal is $\phi^t$ or $\phi$-compatible if and only if it remains compatible after localization at itself, and so we may assume that $J = \bm$ is a maximal $\phi^t$-compatible ideal in a local ring $(R, \bm)$.  But now suppose that $\phi(F^e_* J) \nsubseteq J = \bm$.  In other words that the ideal $\phi(F^e_* J)$ is not contained in the maximal ideal of a local ring, $(R, \bm)$.  Thus $\phi(F^e_* J) = R$ which certainly implies that $\phi^t(F^{te}_* J) = R$ and completes the proof of (c).

We also prove (e) in some detail.  First notice that $\mathfrak{a}$ is necessarily contained in every non-zero $\phi$-compatible ideal (because they are radical).  Since $\phi$ is surjective, there exists $c \in R$ such that $\phi(F^e_* c) = 1$.  For any $z \in \ba$, $c z^{p^e} \in \ba$ as well and so $z = \phi(F^e_* z^{p^e} c) \in \phi(F^e_* \ba)$.  Thus $\ba \subseteq \phi(F^e_* \ba)$.  Repeating this argument yields  the ascending chain in (e).  By the Noetherian hypothesis, this stabilizes say at $C$.  By construction, it is a $\phi$-compatible ideal.  If $J \neq 0$ is any other $\phi$-compatible ideal then $J = \sqrt{J} \supseteq \ba$ by (b) and so the fact that $C$ is minimal follows immediately.
\end{proof}

We point out a particular special case of (b) that we will use later.

\begin{corollary}
Using the notation from Proposition \ref{propFSplittingsProperties}(b), if $\Hom_R(F^e_* R, R)$ is a \emph{free} $F^e_* R$-module with generator $\Phi_R$, then $\phi(F^e_* \bullet) = \Phi_R(F^e_* z \cdot \bullet)$ for some $z \in R$.  In this case $D_{\phi} = \Div(z)$.
\end{corollary}


Consider the following example with regards to Proposition \ref{propFSplittingsProperties}(a).

\begin{example}
\label{ex.PolyRing}
 Set $S = k[x_1, \dots, x_n]$ where $k$ is a perfect field of characteristic $p$.  In this case, $F^e_* S$ is a free $S$-module with basis $\{ F^e_* x_1^{\lambda_1} \dots x_n^{\lambda_n} \}_{0 \leq \lambda_i \leq p^e-1}$.  The map $\Phi_S$ from Proposition \ref{propFSplittingsProperties}(a) is the map which sends the basis element $F^e_* x_1^{p^e - 1} \dots x_n^{p^e - 1}$ to $1$ and all the other basis elements to zero.  We will explain later in Section \ref{sec.FurtherRemarks} how such maps can easily be implemented in a computer.

In this case, the divisor associated to $\Phi_S$ via Proposition \ref{propFSplittingsProperties}(b) is the trivial divisor.  Thus there are no non-trivial $\Phi_S$-compatible ideals by Proposition \ref{propFSplittingsProperties}(b).
\end{example}

We our next goal is to recall Fedder's Lemma;  this will allow us to translate the problem of finding compatible ideals of $R = k[x_1, \dots, x_n]/I$ to finding compatible ideals on $S = k[x_1, \dots, x_n]$.   The point is that if $R = S/I$, then maps $\bar\phi : F^e_* R \to R$ come from maps $\phi : F^e_* S \to S$, which Fedder's Lemma \emph{precisely} identifies.  On the other hand, since $\Phi_S$ generates $\Hom_S(F^e_* S, S)$, we can write $\Phi_S(F^e_* z \cdot \bullet) = \phi(F^e_* \bullet)$.  Thus the choice of $\bar\phi$ is determined by the choice of a certain $z \in S$.  Roughly speaking, Fedder's Lemma says that the set of allowable $z$ are exactly $I^{[p^e]} : I$.



\begin{lemma}[Fedder's Lemma] \cite[Lemma 1.6]{FedderFPureRat}
\label{lem.FeddersLemma}
Suppose that $S = k[x_1, \dots, x_n]$ for some perfect field $k$ and $R = S/I$ for some ideal $I \subsetneq S$ where $\rho : S \to R$ is the canonical surjection.  Then:
\begin{itemize}
\item[(a)] If $\bar\phi : F^e_* R \to R$ is any $R$-linear map, then there exists a $S$-linear map $\phi : F^e_* S \to S$ which is compatible with $I$ such that $\bar\phi = \phi/I$ (making the diagram commute as in Equation \ref{eq.CompatibleDiagram}).
\item[(b)] Given $\bar\phi$ and $\phi$ as in (a), then $\bar\phi = \phi/I$ is surjective if and only if $\phi$ is surjective at all points in a neighborhood of $V(I) \subseteq \Spec S$.  Furthermore, if $\phi(F^e_* \bullet) = \Phi_S(F^e_* z \cdot \bullet)$, then $\phi$ is surjective at a point $\bm \in \Spec S$ if and only if $z \notin \bm^{[p^e]}$.
\item[(c)]  An arbitrary map $\phi \in \Hom_S(F^e_* S, S)$ satisfies $\phi(F^e_* I) \subseteq I$ if and only if there exists $c \in I^{[p^e]} : I$ such that we can write $\phi(F^e_* \bullet) = \Phi_S(F^e_* c \cdot \bullet)$ where $\Phi_S$ is as in Example \ref{ex.PolyRing}.  Combining this with (a), we see that
\[
(F^e_* (I^{[p^e]} : I)) \cdot \Hom_S(F^e_* S, S)
\]
is exactly the set of elements of $\Hom_S(F^e_* S, S)$ which are compatible with $I$.
\item[(d)]  With the notation from (c), there exists an isomorphism:
\[
\Hom_R(F^e_* R, R) \cong \left( (F^e_* (I^{[p^e]} : I)) \cdot \Hom_{S}(F^e_* S, S) \right) \Big/ \left( (F^e_* I^{[p^e]}) \cdot \Hom_S(F^e_* S, S) \right).
\]
\end{itemize}
\end{lemma}
\begin{proof}  The proof found in \cite[Lemma 1.6]{FedderFPureRat} is quite easy to read and so we will not repeat it here.
\end{proof}

We now state an easy corollary of Fedder's Lemma which shows that in fact we may reduce to the case of a Frobenius splitting on $S = k[x_1, \dots, x_n]$.

\begin{corollary}
\label{cor.SplittingsLift}
Using the notation from Lemma \ref{lem.FeddersLemma}(a), additionally suppose that $\bar\phi$ is surjective.  Then there exists a map $\psi : F^e_* S \to S$ such that
\begin{itemize}
\item[(i)]  $\psi$ is a Frobenius splitting.
\item[(ii)]  For every $\bar\phi$-compatible ideal $J \subseteq R$, the inverse image $\rho^{-1}(J) \subseteq S$ is compatibly \emph{split} by $\psi$ (although there may be new $\psi$-compatible ideals not coming from $\bar\phi$).
\end{itemize}
\end{corollary}
\begin{proof}
For simplicity, we use $M$ to denote the $F^e_*S$-module $\Hom_S(F^e_*S, S)$.  Fix any $\phi$ as in Lemma \ref{lem.FeddersLemma}(a) and suppose that $c \in S$ is such that $\phi(F^e_* c) \notin I$, such a $c$ exists by the surjectivity of $\bar\phi$.  Now consider the $F^e_* S$-submodule $W \subseteq M$ generated by $(F^e_* I^{[p^e]}) \cdot M$ and $\phi$, in other words
\[
W := \langle \phi \rangle_{F^e_* S} + (F^e_* I^{[p^e]}) \cdot M.
\]
Note that the formation of $W$ commutes with localization.  Also note that for any $\psi(F^e_* \bullet) = \phi(F^e_* b \cdot \bullet) + \Phi_S(F^e_* f \cdot \bullet) \in W$, the induced map $\psi/I =: \bar\psi$ on $R = S/I$ simply coincides with $\bar\phi(F^e_* \bar{b} \cdot \bullet)$

Consider now the map $\Gamma : W \to S$ defined by the rule $\Gamma(\psi) = \psi(F^e_* 1)$.  We will prove that $\Gamma$ is surjective and thus that $1 \in \Image(\Gamma)$.
It is enough to prove the statement locally at every prime $Q \in \Spec S$.  There are two cases, either $Q \in V(I) \subseteq \Spec S$ or not.
\begin{itemize}
\item[(a)] If $Q \in V(I)$, then after localizing at $Q$, we notice that $\phi_Q : F^e_* S_Q \to S_Q$ has $1$ in its image, say $\phi_Q(a) = 1$.  It follows that the function $\gamma(F^e_* \bullet) = \phi_Q(F^e_* a \cdot \bullet)$ sends $1$ to $1$ and so $\Gamma_Q(\gamma) = 1$.
\item[(b)]  If $Q \notin V(I)$, then it is even easier since then $W_Q = \Hom_{S_Q}(F^e_* S_Q, S_Q)$ and claim follows immediately.
\end{itemize}
Since now $1 \in \Image(\Gamma)$, there exists $\psi \in \Hom_S(F^e_* S, S)$ such that $\psi(F^e_* 1) = 1$.

Suppose finally that $J \subseteq R$ is $\bar\phi$-compatible.  We must show that $J' = \rho^{-1}(J)$ is $\psi$-compatible.  We notice that $\psi$ induces $\psi/I : F^e_* R \to R$ and furthermore, $\psi/I(F^e_* \bullet) = \bar\phi(F^e_* b \cdot \bullet)$ for some $b \in R$ as observed initially.  It is enough to show that $\psi/I(F^e_* J) \subseteq J$ but this is obvious by the above characterization of $\psi/I$.
\end{proof}

As a final remark in this section, we mention a common source of compatibly Frobenius split ideals.

\begin{remark}
Suppose that $X$ is a projective Frobenius split variety projectively normally embedded into $\bP^n$.  In that case the affine cone over $X$ is Frobenius split and so yields a Frobenius split ring (the projective normality then guarantees that it is a quotient of the ring $S = k[x_0, \dots, x_n]$).  This is a particularly common way of producing Frobenius split rings.  In this case, the same Fedder-type Lemma works, and one can compute the compatibly Frobenius split subvarieties of $X$ by computing the compatible ideals of $\Spec S$.
\end{remark}

\section{The statement of the algorithm}
\label{sec.AlgStatement}
Suppose that $S := k[x_1, \dots, x_n]$ where $k$ is a perfect field of characteristic $p > 0$.

Because of Fedder's Lemma \ref{lem.FeddersLemma} and Corollary \ref{cor.SplittingsLift}, we additionally suppose that $\phi : F^e_* S \to S$ is a \emph{surjective} $S$-linear map (for example a Frobenius splitting) that is compatible with $I$ (and henceforth  $I$ will not play much of a role).  In a later section, we will handle the non-surjective case.
The advantage with working with $S$ instead of $R$ is that $F^e_* S$ is a free $S$-module, and so specifying $\phi$ is the same as specifying where a basis is sent.
Finally we fix $z \in S$ such that $\Phi_S(F^e_* z \cdot \bullet) = \phi(F^e_* \bullet)$ where $\Phi_S$ is as in Example \ref{ex.PolyRing}.

With the notation above, and given any prime ideal $Q \subseteq S$ (in practice containing $I$) the recursive algorithm described below produces a list of all prime $\phi$-compatible ideals which properly contain $Q$.  In order to do this, the algorithm finds the smallest $\phi$-compatible ideal properly containing $Q$.  The initial input $Q$ to the algorithm can be the zero ideal or each of the minimal primes of $I$.

Of course, the plan of the algorithm is to apply Proposition \ref{propFSplittingsProperties}(e) to the ring $S/Q$, see Step (3) below.  Therefore we need to define the ideal $\ba$ which all the Frobenius split subvarieties contain.  Identifying the singular locus is straightforward, see step (1) below, and identifying the locus corresponding to the divisor $D_{\phi}$ on $R$ is accomplished in (2).  Step (4) is then recursive where we replace $Q$ by larger ideals.

Here are the steps of the algorithm (later we will describe an algorithm which works for non-surjective $\phi$).
\vskip 6pt
\begin{itemize}
\item[(1)]  Find an ideal $J \subseteq S$ such that $\Spec(S/Q) \setminus (V(J) \cap \Spec(S/Q))$ is a regular non-empty scheme.  For example $J$ could define the singular locus of $S/Q$.
\item[(2)]  Compute $B := \Ann_S \left( (Q^{[p^e]} : Q) / ( \langle z \rangle + Q^{[p^e]}) \right) = \left( ( \langle z \rangle + Q^{[p^e]}) :  (Q^{[p^e]} : Q) \right).$  This ideal in not contained in $Q$ as long as $\phi/Q$ is non-zero.
\item[(3)]  Find the first $t$ such that $\phi^{t}( F^{te}_* (JB + Q) ) = \phi^{t+1}(F^{(t+1)e}_* (JB + Q) ) = C$ (as in Proposition \ref{propFSplittingsProperties}(e)).  We will show that the ideal $C$ properly contains $Q$.
\item[(4)]  Let $Q_1, \dots, Q_m$ be the minimal primes of $C$, add them to the list of compatible ideals, and repeat the algorithm with $Q_i = Q$.  Every $\phi$-compatible ideal, properly containing $Q$, that is minimal with respect to inclusion, appears in the list of the $Q_i$.
\end{itemize}


\vskip 6pt
In Section \ref{sec.FurtherRemarks} below, we will discuss issues of complexity, especially with regards to the integer $t$ needed in Step (3),

\section{The proof that the algorithm works}

In this section, we prove the individual claims made in the algorithm.
Step (1) has no associated claims so we move on to Step (2) which is the only technical point in the argument.  Let us briefly explain the difficulty.  While the divisor associated to $\phi$ is easy to identify on $\Spec S$ (it is merely $\Div(z)$), it is harder to identify on $S/Q$ (which need not even be normal), indeed, this divisor has no transparent relation to $\Div(z)$ except in special cases.  Step (2)  identifies an ideal whose support contains $D_{\phi/Q}$ at least on the locus where $S/Q$ is normal.  To do this we utilize Fedder's Lemma \ref{lem.FeddersLemma}.

Let us informally explain where the ideal $B$ comes from.  Indeed, $(Q^{[p^e]} : Q)$ corresponds via Fedder's Lemma to the set of all elements of $\Hom_S(F^e_* S, S)$ which are compatible with $Q$.  On the other hand, $\langle z \rangle + Q^{[p^e]}$ corresponds to the submodule of $\Hom_S(F^e_* S, S)$ generated by all $\psi$ which are both compatible with $Q$ and such that $\psi/Q = \phi/Q$.  We then observe that
\[
\left( (\langle z \rangle + Q^{[p^e]}) :  (Q^{[p^e]} : Q) \right)
\]
simply defines the locus where these two modules are distinct.  On the regular locus of $R/Q$, this is simply the defining equation of $D_{\phi/Q}$.

Now we carefully prove a generalization of the claim from step (2).

\begin{lemma}
\label{lem.ProofOf2}
For any $\phi : F^e_* S \to S$ compatible with $Q$, the ideal $B$ constructed in (2) is not contained in $Q$ as long as $\phi(F^e_* S)$ is not contained in $Q$ (we do not require that $\phi$ is surjective).
\end{lemma}
\begin{proof}
Consider the induced $S/Q$-linear map $\phi/Q : F^e_* (S/Q) \to S/Q$.  Notice that $\Image(\phi/Q) = \phi(F^e_* S)/Q \neq 0$ since $Q$ does not contain $\phi(F^e_* S)$.  Therefore, the map $\phi/Q : F^e_*(S/Q) \to S/Q$ is not the zero map.  We will use $N$ to denote the $F^e_*S$-module $\Hom_S(F^e_* S, S)$ and use $M$ to denote the $F^e_* (S/Q)$-module $\Hom_{S/Q}(F^e_* (S/Q), S/Q)$.  Consider the cyclic $F^e_* (S/Q)$-submodule of $M$
\[
K := \langle \phi/Q \rangle \subseteq M := \Hom_{S/Q}(F^e_* (S/Q), S/Q).
\]
Now,
$M$ is a $F^e_* (S/Q)$-module of \emph{generic} rank 1, and so there exists an element $F^e_* d \in F^e_* S \setminus F^e_* Q$ such that
\[
(F^e_* d) \cdot M \subseteq K.
\]
We will show that $d \in F^e_* B$ which will imply that $B$ is not contained in $Q$.

We use Fedder's Lemma \ref{lem.FeddersLemma} to see that
\[
\Hom_{S/Q}(F^e_* (S/Q), S/Q) = M \cong \left( F^e_* (Q^{[p^e]} : Q) \cdot N\right)  \Big/ \left( (F^e_* Q^{[p^e]}) \cdot N\right)
\]
Furthermore, under this identification, the submodule $K$ corresponds to
\[
\left( F^e_* (\langle z \rangle + Q^{[p^e]}) \cdot N \right) \Big/ \left( (F^e_* Q^{[p^e]}) \cdot N \right).
  \]
It follows that $d$ multiplies $F^e_* (Q^{[p^e]} : Q)$ into $F^e_* (\langle z \rangle + Q^{[p^e]})$ and so $d \in B$ as desired.
\end{proof}

The remaining claims from the algorithm are easy and we now prove the claim in (3).

\begin{lemma}
\label{lemmaStep3}
The ideal $C$ constructed in (3) exists and properly contains $Q$.
\end{lemma}
\begin{proof}
 By (1) and (2), the ideal $JB$ is not contained in $Q$, thus $JB + Q$ properly contains $Q$.  The result then follows since this chain of ideals stabilizes by Proposition \ref{propFSplittingsProperties}(e).
\end{proof}

Finally, we prove the claim in step (4).

\begin{lemma}
With the notation of step (4), if $P$ is any prime $\phi$-compatible ideal, properly containing $Q$, then $P$ contains $Q_i$ for some $1 \leq i \leq m$.
\end{lemma}
\begin{proof}
This follows immediately from Proposition \ref{propFSplittingsProperties}(e).
\end{proof}

\section{A more general algorithm}
\label{sec.GeneralizedAlg}

In this section we describe a slight modification of this algorithm which produces something of interest even if $\phi$ is not surjective.
Fix $S$ and  $\phi : F^e_* S \to S$ an $S$-linear map as in section \ref{sec.AlgStatement}, but do not assume that $\phi$ is surjective.  Consider the ideal of $S$, $K = \sqrt{ \Image(\phi) }$.  With a minor change to the algorithm presented above, we can compute all the prime $\phi$-compatible ideals of $J$ not containing $K$.  Furthermore, we can identify all the $\phi$-compatible primes, see Remark \ref{rem.FindingBoringPrimes}.

Obviously there is another way to do this too, set $V(K) \subseteq \Spec S$ to be the vanishing locus of $K$.  Then one can cover $\Spec R \setminus V(K)$ with affine charts, charts where $\phi$ is surjective, and compute the compatible ideals on each of those charts.  However, a minor change of our original algorithm allows us to run it without this complication.

The input for each stage of our algorithm is the same as before and the steps are quite similar (listed below):
\vskip 6pt
\begin{itemize}
\item[(1*)]  Find an ideal $J \subseteq S$ such that $\Spec(S/Q) \setminus (V(J) \cap \Spec(S/Q))$ is a regular non-empty scheme.  For example $J$ could define the singular locus of $S/Q$.
\item[(2*)]  Compute $B := \Ann_S \left( (Q^{[p^e]} : Q) / ( \langle z \rangle + Q^{[p^e]}) \right) = \left( \langle z \rangle + Q^{[p^e]}) :  (Q^{[p^e]} : Q) \right).$  This ideal in not contained in $Q$ as long as $Q$ does not contain $K$.
\item[(3*)]  Define an ascending chain of ideals of $S$ recursively as follows:  $C_0 = JB + Q$ and $C_t = \phi(F^e_* C_{t-1}) + C_{t-1}$.  Find the first $t$ such that $C_t = C_{t+1}$ and set $C = C_t$.  The ideal $C$ properly contains $Q$.
\item[(4*)]  Let $Q_1, \dots, Q_m$ be the minimal primes of $C$ which do not contain $K$.  Then repeat the algorithm with $Q_i = Q$ (all of the $Q_i$ will be $\phi$-compatible ideals so should be added to the list of valid outputs).  Every $\phi$-compatible ideal, properly containing $Q$ but not $K$, that is minimal with respect to inclusion, appears in the list of the $Q_i$.
\end{itemize}

\vskip 12pt
Running this recursively will produce all $\phi$-compatible primes not containing $K$.

Now we prove that the algorithm is correct.  Again property (1*) has no associated claims.
Property (2*)'s proof is already contained in the proof of Lemma \ref{lem.ProofOf2}.
We now prove that the assertions in properties (3*) and (4*) hold.

\begin{lemma}
 The ideal $C$ defined in (3*) properly contains $Q$.
\end{lemma}
\begin{proof}
This chain also is ascending and is thus eventually constant.  It also properly contains $Q$ since $C_0$ does.
\end{proof}

\begin{lemma}
 With the notation of step (4*), if $P$ is any prime $\phi$-compatible ideal, properly containing $Q$, then $P$ contains $Q_i$ for some $1 \leq i \leq m$.
\end{lemma}
\begin{proof}
Note that $C$ vanishes on the same locus as the test ideal $\tau(S/Q, \phi/Q)$ by \cite{SchwedeFAdjunction} and Proposition \ref{propFSplittingsProperties}.  In particular, $P$ contains $\sqrt{ \tau(S/Q, \phi/Q) } \supseteq (JB + Q)/Q = C_0$ because the test ideal is the unique smallest $\phi/Q$-compatible ideal.  Thus $P = \phi(F^e_* P) + P \supseteq \phi(F^e_* C_0) + C_0 = C_1$ and so recursively, $P$ contains $C$ and thus also $P$ contains $\sqrt{C}$ and so it also contains a minimal prime of $\sqrt{C}$.
\end{proof}

\begin{remark}
\label{rem.FindingBoringPrimes}
Finally we explain how to find all the $\phi$-compatible primes.  We have just found all the $\phi$-compatible ideals not containing $K$.  On the other hand, suppose $P \supseteq K$ is a prime ideal, we will show it is \emph{always} $\phi$-compatible.  It is easy to see that $P$ is $\phi$-compatible if and only if $P/K$ is $\phi/K$ compatible.  But $P/K$ is clearly compatible since $\phi/K$ is zero.
\end{remark}

\section{Further remarks}
\label{sec.FurtherRemarks}
In this section we briefly discuss issues related to the implementation of this algorithm and also discuss some connections with previous work.

\subsection{Notes on implementation and complexity}
\label{Section: Notes on implementation}

In this subsection, we briefly discuss the issues surrounding the implementation of this algorithm.

The computationally intensive steps in the algorithm involve computation of the singular locus in step (1), the computation of colon ideals in step (2), the repeated application of $\phi$ in (3) and finally primary decomposition in step (4).   Indeed, all the steps (1), (2) and (4) are already implemented in many computer algebra systems (for example Macaulay2).  Step (1) is simply the computation of, potentially many, minors (although any single minor that doesn't vanish on the given $Q$ would suffice).  The computation of the colons of ideals in step (2) reduces to the computations of ideal intersections as described in \cite[Section 1.8.8]{GreuelPfisterASingularIntroduction}.   Step (3) will be handled below and the primary decomposition in step (4), a list of references for the history of computing primary decomposition can be found on \cite[Page 206]{CoxLittleOSheaIdealsVarietiesAlgorithms}.

In order to implement the algorithm (or the generalized version from Section \ref{sec.GeneralizedAlg}), we need to answer the following question.

\begin{question}
How does one compute the images of the $S$-linear maps $\phi : F^e_* S \to S$?
\end{question}

We fix a $\phi=(F^e_* u) \cdot \Phi_S \in \Hom_S(F^e_* S, S)$.
Given any ideal $J \subseteq S$, $\Phi_S(F^e_* J)$ is an ideal that has appeared previously in several contexts.  Again, $\Phi_S$ is as in example \ref{ex.PolyRing}.
Notably, in \cite{BlickleMustataSmithDiscretenessAndRationalityOfFThresholds} and \cite{KatzmanParameterTestIdealOfCMRings}, also see \cite[Proposition 3.10]{BlickleSchwedeTakagiZhang}, it was shown that
$\Phi_S(F^e_* J)$ is the unique smallest ideal $A \subseteq S$ with the property that $J \subseteq A^{[p^e]}$.
In these contexts it was denoted by $I^{[1/p^e]}$ and $I_e(J)$ respectively.
In the context described, where $S$ is a free module over $S^{p^e}$,
this ideal is highly computable as we next show (cf.~\cite[Section 5]{KatzmanParameterTestIdealOfCMRings} and ~\cite[Proposition 2.5]{BlickleMustataSmithDiscretenessAndRationalityOfFThresholds}).

It is easy to see that for ideals $J_1, J_2 \subseteq S$, we have $I_e(J_1+J_2)=I_e(J_1)+I_e(J_2)$, so the calculation of
$I_e(J)$ reduces to the case where $J$ is generated by one element $g\in S$.
Let $\alpha$ denote $n$-tuple of non-negative integers $(\alpha_1, \dots, \alpha_n)$, let $0\leq \alpha <p^e$ denote the condition
$0\leq \alpha_1, \dots, \alpha_n<p^e$ and
write $x^\alpha=x_1^{\alpha_1} \cdots x_n^{\alpha_n}$.
Now write
\[g=\sum_{b\in \mathcal{B}, \atop{0\leq \alpha<p^e}}
g_{b,\alpha}^{p^e} b x^\alpha\]
where $g_{b,\alpha}\in S$.
We claim that $I_e(\langle g \rangle)$ is the ideal $A$ generated by
$\left\{g_{b, \alpha} \,|\, b\in \mathcal{B}, 0\leq \alpha < p^e \right\}$.
To see this note first that, clearly, $g\in A^{[p^e]}$.
If $L\subseteq S$ is such that $g\in L^{[p^e]}$ then we can find
$a_1,\dots, a_s\in L$ and $r_1, \dots, r_s\in S$ such that
$$g=\sum_{b\in \mathcal{B}, \atop{0\leq \alpha< p^e}} g_{b,\alpha}^{p^e} b x^{\alpha}  =\sum_{i=1}^s r_i a_i^{p^e} .$$
For all $1\leq i\leq s$ we can now write
$$r_i =\sum_{b\in \mathcal{B}, \atop{0\leq \alpha < p^e}} r_{b,\alpha,i}^{p^e} b x^\alpha$$
where $r_{b,\alpha,i}\in S$ and we obtain
$$\sum_{b\in \mathcal{B}, \atop{0\leq \alpha< p^e}} g_{b,\alpha}^{p^e} b x^\alpha=
\sum_{b\in \mathcal{B}, \atop{0\leq \alpha< p^e}} \left(\sum_{i=1}^s r_{b,\alpha,i}^{p^e} a_i^{p^e}\right) b x^\alpha .$$
Since these are direct sums, we may compare coefficients and deduce that
for all $b\in \mathcal{B}$ and $0\leq \alpha<p^e$, $g_{b, \alpha}^{p^e}=\sum_{i=1}^s r_{b,\alpha,i}^{p^e} a_i^{p^e}$ hence
$g_{b, \alpha}=\sum_{i=1}^s r_{b, \alpha,i} a_i$
and $g_{b,\alpha} \in L$.

This construction translates easily into an algorithm as follows.
Extend the ring $S$ to $T=S[y_1, \dots, y_n]$, chose a term ordering in which the variables $x_1, \dots, x_n$
are bigger than $y_1, \dots, y_n$ and reduce the $g\in T$ with respect to the ideal generated by $x_1^{p^e}-y_1, \dots, x_n^{p^e}-y_n$
to obtain $g\equiv \sum_{0\leq \alpha < p^e} g_\alpha x^\alpha$ where $g_\alpha\in k[y_1, \dots, y_n]$.
For each $0\leq \alpha < p^e$ we can write
$g_\alpha$  as a sum of terms $\lambda_1 y^{\beta^{(1)}} + \dots + \lambda_m y^{\beta^{(m)}}$ with
$\lambda_1, \dots, \lambda_m\in k$ and for each $1\leq i \leq m$ we can write
$\lambda_i=\sum_{b\in \mathcal{B}} \lambda_{i,b}^{p^e} b$ where $\lambda_{i,b}\in k$.
Now $g_\alpha=\sum_{b\in \mathcal{B}} \sum_{i=1}^m \lambda_{i,b}^{p^e} b y^{\beta^{(i)}}$
and $I_e(g)$ is the ideal generated by $\sum_{i=1}^m \lambda_{i,b}  x^{\beta^{(i)}}$ for all choices of
$0\leq \alpha<p^e$ and $b\in \mathcal{B}$.
Notice that the complexity of applying $\phi$ is essentially the complexity of finding the $S^p$-coordinates
of elements in $S$ in terms of a given set of free generators of $S$.

One can also ask the following.

\begin{question}
How many times must $\phi$ be applied in step (3)?
\end{question}

In the experiments we have done so far, the condition in step (3) does not seem to be a limiting factor since applying $\Phi_S$ is itself a very fast operation and the required value of $t$ is quite small.  Of course, the particular $t$ needed also depends upon the given ideal $Q$.

Indeed, one can give a reasonable bound on $t$ based upon the sort of analysis found in \cite{BlickleMustataSmithDiscretenessAndRationalityOfFThresholds}.  Fix $\ba$ an ideal in $S = k[x_1, \dots, x_n]$ and write $\phi(F^e_* \bullet) = \Phi_S(F^e_* z \cdot \bullet)$.  Suppose that $\ba$ is generated by elements of degree at most $d$ in $S$ so that $z \cdot \ba$ is generated by elements of degree at most $(\deg z) + d$.  We now consider the generators of $F^e_* (z \cdot \ba) \subseteq F^e_* S$ as an \emph{$S$-module}.  It is easy to see that $F^e_* (z \cdot \ba)$ is generated by elements of degree at most $(\deg z) + d + (p^e - 1)n$ (since $F^e_* S$ is as an $S$-module is generated by elements of degree at most $(p^e - 1)n$).  Note that if $f$ is a polynomial of degree $k$, then it follows that $\Phi_S(F^e_* f)$ is a polynomial of degree at most $\lfloor \frac{k - (p^e - 1)n}{p^e} \rfloor$.  We the see that $\Phi_S(F^e_*  z \cdot \ba)$ is generated by elements of degree at most
\[
\Big\lfloor \frac{(\deg z) + d + (p^e - 1)n - (p^e - 1)n}{p^e} \Big\rfloor = \Big\lfloor \frac{(\deg z) + d}{p^e} \Big\rfloor.
\]
If we apply $\phi$ again, we obtain an ideal generated by elements of degree at most
\[
\Bigg\lfloor \frac{(\deg z) + \Big\lfloor \frac{(\deg z) + d}{p^e} \Big\rfloor}{p^e} \Bigg\rfloor \leq \Big\lfloor \frac{(1 + p^e)(\deg z) + d}{p^{2e}} \Big\rfloor
\]
For $t \gg 0$, $\phi^t(F^{te}_* \ba)$ is generated by elements of degree at most
\[
\Big\lfloor \frac{(1 + p^e + \dots + p^{(t-1)e})(\deg z) + d}{p^{te}} \Big\rfloor \leq \Big\lfloor \frac{(p^{te} - 1)(\deg z)}{p^{(t+1)e}} \Big\rfloor + 1\leq \Big\lfloor \frac{\deg z}{p^{e}}\Big\rfloor + 1.
\]
We now ask how many times need we to apply $\phi$ before we reach this stable degree (which is a vector space of polynomials of bounded degree).  But for this we merely need $d/p^{te} \leq 1$, or in other words after at most $\lceil \log_{p^e}(d) \rceil$ applications of $\phi$.

Of course, we still may need to apply $\phi$ further when are working within this vector space.  However, as soon as the containment $\phi^{t}(F^{te}_* z \cdot \ba) \supseteq \phi^{t+1}(F^{(t+1)e}_* z \cdot \ba)$ is equality, that step in our algorithm terminates.  In particular, we are working within this fixed vector space which has dimension at most the binomial coefficient:
\[
M := \binom{\Big\lfloor \frac{\deg z}{p^{e}}\Big\rfloor + 1 + n }{ n }.
\]
Recall that $n$ is the number of variables in $S$.  Therefore, we require at most $M$ applications of $\phi$.

In summary, in step (3), we need only apply $\phi$ at most:
\[
\lceil \log_{p^e}(d) \rceil + \binom{\Big\lfloor \frac{\deg z}{p^{e}}\Big\rfloor + 1 + n }{ n }
\]
times.

\subsection{Connections with previous work}

We now also briefly explain some of the other ways the ideas in this algorithm have previously appeared.  We use the notation from Section \ref{Section: Notes on implementation}, in particular $\phi(\bullet) = \Phi_S(u \cdot \bullet)$.

\begin{definition} \cite[Definition 5.5]{KatzmanParameterTestIdealOfCMRings}, \cf \cite{BlickleMustataSmithDiscretenessAndRationalityOfFThresholds}
As above, fix $u \in R$.  For any ideal $J\subseteq S$ we define  $J^{\star^e u}$ to be the smallest ideal $A\subseteq S$ containing $J$ with the property $u A \subseteq A^{[p^e]}$.
\end{definition}

It follows that the sequence of ideals $\{ C_i \}_{i\geq 0}$ from step (3*) stabilizes at the value $\left(JB+Q\right)^{\star^e u}$.  This construction (and the fact that it had already been computed by the first author) was part of the inspiration for this project.  In view of this, step (3*) can be totally phrased in the language of $J^{\star^e u}$.

The original motivation for studying $J^{\star^e u}$ was to compute the test ideal.  In particular, the ideal $C$ we construct in (3*) restricts to the test ideal $\tau(S/Q, \phi/Q)$ in many cases (with additional work, it can always be made to restrict to the test ideal, see for example \cite{KatzmanParameterTestIdealOfCMRings}); however, it always restricts to the test ideal $\tau(S/Q, \phi/Q)$ if $\phi/Q$ is surjective.




\begin{remark}
One can also view the results of this paper from the point of view of Frobenius maps on the injective hull of residue fields. Let $(S, m)$ be a complete local regular
ring and $E=E_S(S/m)$ is the injective hull of its residue field. Let $f: S \rightarrow S$ be the Frobenius map
$f(s)=s^p$, and let $S[\Theta; f^e]$ be the skew-polynomial ring with coefficients in $S$ where the variable $\Theta$ satisfies $\Theta s = f^e(s) \Theta$
for all $s\in S$.

The $S$-module $E$ has a natural structure of $S[T; f^e]$-module which can be described by identifying $E$ with a module of inverse polynomials
$k[x_1^-, \dots, x_n^-]$ (cf.~\cite[Example 12.4.1]{BrodmannSharpLocalCohomology}) and extending additively the action
$T \lambda x_1^{-\beta_1}  \cdots x_n^{-\beta_n}  = \lambda^{p^e} x_1^{-p^e \beta_1}  \cdots x_n^{-p^e \beta_n} $ for all
$\lambda\in k$ and $\beta_1, \dots, \beta_n>0$.
Any structure of $S[\Theta; f^e]$-module on $E$ is given by $\Theta=u T$, where $T$ is the natural action above, and, with this $\Theta$,
an $S$-submodule $\Ann_{E} J \subseteq E$ is an $S[\Theta; f^e]$-submodule if and only if $u J \subseteq J^{[p^e]}$ (cf.~\cite[section 4]{KatzmanParameterTestIdealOfCMRings}).
Thus we see that the $u T$-compatible ideals of $R=S/I$ are the annihilators of $S[\Theta; f^e]$-submodules of $E$ which contain $I$, i.e.,
$S[\Theta; f^e]$-submodules of $\Ann_E I=E_R(R/mR)$ and hence our algorithm produces these.
These annihilators form the set of \emph{special ideals} in the language of
 \cite{SharpGradedAnnihilatorsOfModulesOverTheFrobeniusSkewPolynomialRing} and
 \cite{KatzmanParameterTestIdealOfCMRings}.
An analysis of our algorithm shows that it will produce all special primes $P$ for which the restriction of $\Theta$ to $\Ann_{E} P$ is not
the zero map.
\end{remark}

\section{The algorithm in action}

In this section we present some interesting calculations performed with a Macaulay2 implementation of the algorithms presented in this paper.

First we include an example where we step through the algorithm.  This example is also interesting because the ideal defining the singular locus of $R = S/I$ is not always compatible with our choice of $\phi$.
\begin{example}
Consider the ring $S = k[x,y,z,w]$ where $k$ is any perfect field of characteristic $3$ and set $I = \langle x^2 - yz \rangle$.

We set
\[
z = (x^2 - yz)^{2}w^2x(x+1) = x^6 w^2 + x^4 yzw^2 + x^2 y^2 z^2 w^2 + x^5 w^2 + x^3 yzw^2 + xy^2z^2w^2
 \]
and fix $\phi(F_* \bullet) = \Phi_S(F_* z \cdot \bullet)$.  It is easy to see that $\phi$ is $I$-compatible by Fedder's Lemma \ref{lem.FeddersLemma} noting that $z$ is a multiple of $(x^2 w - yzw)^{2}$ .  Furthermore, note that $z$ has a term $x^2y^2z^2w^{2} \notin \bm^{[3]} = \langle x^3, y^3, z^3, w^3 \rangle$, which implies that $\phi$ is surjective at the origin.  More generally, the same term implies that $\phi(F_* 1) = 1$ which means that $\phi$ is surjective everywhere as it is a Frobenius splitting.

First we set $Q= I$
In step (1) of the algorithm, we compute the singular locus of $S/Q$, it is defined by the ideal $J = \langle x,y,z \rangle$.  For step (2) of the algorithm, Macaulay2 will easily verify that $B := (z + Q^{[3]}) : (Q^{[3]} : Q) = \langle x^2 - yz, x^2w^2 + xw^2 \rangle$. Now we move on to step (3): one can either verify by hand, or by Macaulay2 that
\[
\phi(F_* JB+Q) = \langle w, x^2 - yz \rangle.
\]
On the other hand
\[
\phi^2\left(F^2_* (JB + Q)\right) = \phi(F_* \phi(F_* \langle w, x^2 -yz \rangle)) = \langle w, x^2 - yz \rangle
\]
as well, so the ideal $C = \langle w, x^2 - yz \rangle$.  This ideal is already prime so no primary decomposition is needed for step (4).

Now we repeat the algorithm with a new $Q' := \langle w, x^2 - yz \rangle$.  In step (1), the singular locus is now $J' = \langle x,y,z,w \rangle$.  In step (2), one obtains $B' := \langle w, yz + x, x^2 + x\rangle$.  For step (3), we easily compute that
\[
\phi(F_* J'B'+Q') = S = \phi^2(F^2_* J'B'+Q')
\]
In particular, there are no new $Q_i$ and the algorithm terminates.

It follows that the only proper non-zero $\phi$-compatible ideal of $S$ properly containing $I$ is $\langle w, x^2 - yz \rangle_S$.  In particular, the only proper non-zero $\phi/I = \bar\phi$-compatible ideal of $R = S/I$ is the ideal $\langle w \rangle_{R}$.
\end{example}

Now we perform a more involved calculation which looks in greater detail at the example given in \cite[section 9]{KatzmanParameterTestIdealOfCMRings}.

\begin{example}
Let $\mathbb{K}$ be the field of two elements, $S=\mathbb{K}[x_1, x_2, x_3, x_4, x_5]$, denote $\mathfrak{m}=\langle x_1, x_2, x_3, x_4, x_5 \rangle$.

Now take $I$ be the ideal of $S$ generated by the $2\times 2$ minors of
$$\left( \begin{array}{llll} x_1 & x_2 & x_2  & x_5 \\ x_4 & x_4 & x_3 & x_1 \end{array} \right)$$
and let $S=R/I$.
We consider a $\phi\in\Hom_R(F^1_*S,S)$ induced by pre-multiplying the $R$-linear map $\Phi_S$ from Example \ref{ex.PolyRing} by an element $z \in (I^{[p]} : I) \subseteq S$
\[
z = {x}_{1}^{3} {x}_{2} {x}_{3}+{x}_{1}^{3} {x}_{2} {x}_{4}+{x}_{1}^{2} {x}_{3} {x}_{4} {x}_{5}+{x}_{1} {x}_{2} {x}_{3} {x}_{4} {x}_{5}+{x}_{1} {x}_{2} {x}_{4}^{2}
{x}_{5}+{x}_{2}^{2} {x}_{4}^{2} {x}_{5}+{x}_{3} {x}_{4}^{2} {x}_{5}^{2}+{x}_{4}^{3} {x}_{5}^{2}
\]
Since $z \notin \langle x_1^2, x_2^2, x_3^2, x_4^2, x_5^2 \rangle = \bm^{[2]}$, we see that $\phi$ is surjective from Fedder's Lemma, Lemma \ref{lem.FeddersLemma}(b).

Our algorithm now produces a complete set of $\phi$-compatible primes as follows:
$$ R, \langle x_1, x_4 \rangle, \langle x_1, x_4, x_5 \rangle$$
$$ \langle x_1+x_2, x_1^2+x_4x_5\rangle , \langle x_1+x_2, x_1^2+x_4x_5\rangle , \langle x_3+x_4, x_1+x_2, x_2^2+x_4x_5\rangle ,$$
$$ \langle x_1, x_2, x_5 , x_3+x_4\rangle , \langle x_1, x_2, x_4\rangle , \langle x_1, x_2, x_5\rangle , \langle x_1, x_3, x_4\rangle ,$$
$$ \langle x_1, x_2, x_3, x_4\rangle , \langle x_1, x_2, x_4, x_5\rangle , \langle x_1, x_3, x_4, x_5\rangle , \mathfrak{m}$$

Consider now the Frobenius action $\Theta=zT$ on the injective hull $E$ of the residue field of
$\mathbb{K}[\![x_1, x_2, x_3, x_4, x_5]\!]$. The $\phi$-compatible primes above are also the
the special primes of the $S[\Theta; f]$-module $E$.
In \cite[section 9]{KatzmanParameterTestIdealOfCMRings} it was shown that
there is a $S[\Theta; f]$-linear surjection of $E$ onto $\HH^2_{\mathfrak{m}S}(S)$
where the latter is equipped with its canonical $S[\Theta; f]$-module structure, and
the set of  prime annihilators of $S[\Theta; f]$-submodules of  this quotient consist of the first three $\phi$-compatible primes above.

\end{example}

\begin{example}
We now consider an example suggested by the referee and inspired by the calculation of Schubert varieties, see in particular  \cite[section 7]{KnutsonFrobeniusSplittingPointCountingAndDegeneration} for the origin of the element $u$ below.

Let $S$ to be a polynomial
ring in indeterminates $\{ x_{i j} \,|\, 1\leq j < i \leq 4 \}$
over a field $\mathbb{K}$ of prime characteristic $2$ and let
$$u=x_{41} (x_{31} x_{42}-x_{41} x_{32}) (x_{41}-x_{21} x_{42}-x_{31} x_{43}+x_{21} x_{32} x_{43});$$
this is the product of the four lower left minors of the matrix
$$M=
\left[
\begin{array}{cccc}
1 & 0 & 0 & 0\\
x_{21} & 1 & 0 & 0\\
x_{31} & x_{32} & 1 & 0\\
x_{41} & x_{42} & x_{43} &1
\end{array}
\right] .
$$

For  sets $\alpha, \beta\subseteq \{1,2,3,4\}$ of the same cardinality, let $[\alpha, \beta]$ denote the determinant of the submatrix  of  $M$ consisting of the rows in $\alpha$ and the columns in $\beta$.
Our algorithm produces the following twenty-three compatible ideals:
the ideal generated by all variables, twelve  ideals generated by the variables in the positions\\
\MatrixDiagram{ '"7"  '"9"  '"12" '"a3" "7"}
\MatrixDiagram{ '"9"  '"14"  '"b2" '"7" "9"}
\MatrixDiagram{ '"9"  '"5"  '"a2" '"12" "9"}
\MatrixDiagram{ '"9"  '"3"  '"a1" '"12" "9"}
\MatrixDiagram{ '"9"  '"5"  '"a2" '"a3" '"b2" '"14" "9"}
\MatrixDiagram{ '"9"  '"7"  '"c1" '"16" "9"}\\
\MatrixDiagram{ '"9"  '"3"  '"a1" '"a3" '"b2" '"14" "9"}
\MatrixDiagram{ '"9"  '"5"  '"a2" '"a3" '"c1" '"16" "9"}
\MatrixDiagram{ '"9"  '"5"  '"b1" '"14" "9"}
\MatrixDiagram{ '"9"  '"3"  '"a1" '"a2" '"b1" '"b2" '"14" "9"}
\MatrixDiagram{ '"9"  '"5"  '"b1" '"b2"  '"c1" '"16" "9"}
\MatrixDiagram{ '"9"  '"3"  '"a1" '"a3" '"c1" '"16" "9"}\\
together with the ten ideals\\
$
\langle [34,12]\rangle, \
\langle [34,12], [34,13], [34,23]\rangle,\
\langle [34,12], [23,12], [24,12]\rangle, \\
\langle [34,12], [23,12], [34,13], [34,23], [24,12]\rangle,\
\langle [234,123]\rangle,\
\langle [234,123], x_{41}\rangle,\\
\langle [23,12],x_{41},x_{42}\rangle,\
\langle [23,12],x_{41},x_{42}, x_{43}\rangle, \\
\langle [34,23],x_{31},x_{41}\rangle,\
\langle [34,23],x_{21},x_{31},x_{41}\rangle,\
$

\end{example}

Finally, we compute an example when the map $\phi$ is not surjective to illustrate the more general algorithm.

\begin{example}
We fix a $2\times 4$ matrix of indeterminates, and we let $S$ to be a polynomial
ring in these indeterminates over a field $\mathbb{K}$ of prime characteristic $2$.

For any $1\leq i< j \leq 4$ we denote  $\Delta_{i j}$ the $2\times 2$ minor obtained from columns $i$ and $j$
and for any subset $A\subseteq \{1,2,3,4\}$ we denote $V_A$ the ideal generated by the matrix entries in all columns listed in $A$.


For  $z= \Delta_{1 2} \Delta_{1 3} \Delta_{1 4} = (x_{11}x_{22} - x_{21}x_{12})(x_{11} x_{23} - x_{21}x_{13})(x_{11} x_{24} - x_{21}x_{14})$, we form $\phi(F_* \bullet) = \Phi_S(F_* z \cdot \bullet)$. This $\phi$ is easily seen to not be surjective.  
Our generalized algorithm produces two sets of compatible ideals.
The first is the poset
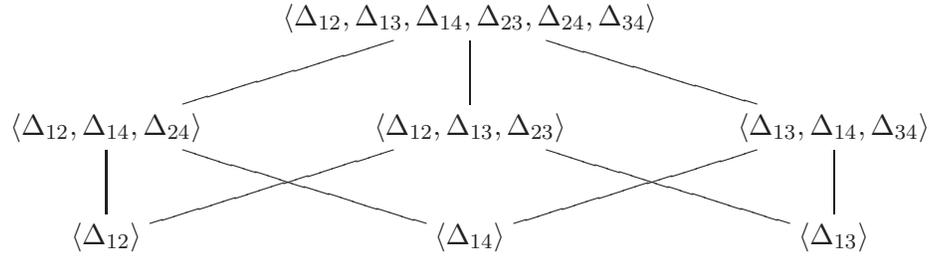
\begin{equation*}
\xymatrix{
 & \langle\Delta_{1 2}, \Delta_{1 3}, \Delta_{1 4}, \Delta_{2 3}, \Delta_{2 4}, \Delta_{3 4}\rangle & \\
\langle\Delta_{1 2}, \Delta_{1 4}, \Delta_{2 4}\rangle  \ar@{-}[ur] &  \langle\Delta_{1 2}, \Delta_{1 3}, \Delta_{2 3}\rangle \ar@{-}[u]& \langle\Delta_{1 3}, \Delta_{1 4}, \Delta_{3 4}\rangle \ar@{-}[ul]\\
\langle \Delta_{1 2} \rangle \ar@{-}[u] \ar@{-}[ur] & \langle \Delta_{1 4}\rangle \ar@{-}[ur] \ar@{-}[ul] &\langle \Delta_{1 3} \rangle \ar@{-}[u] \ar@{-}[ul]\\
}
\end{equation*}

It is easy to see that $z \subseteq V_1^{[2]}$ and so the induced map $\phi/V_{1}$ is the zero map since $\phi(F_* S) \subseteq V_{1}$.  In particular, our algorithm also produces the following poset of primes.

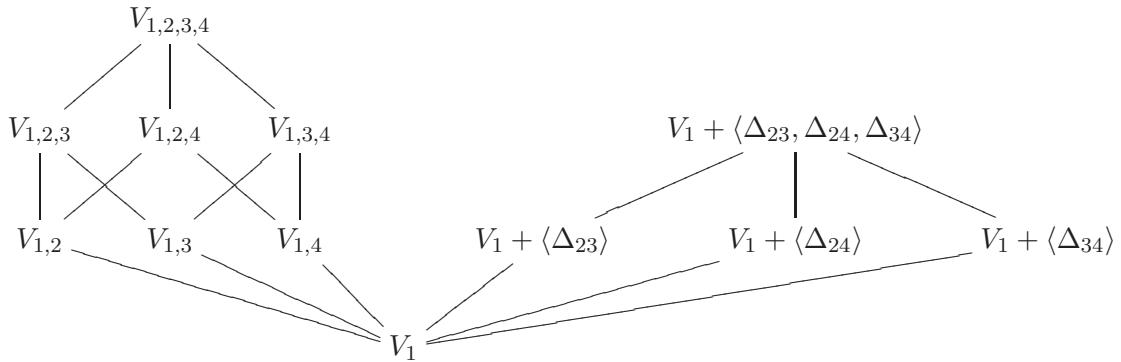
\begin{equation*}
\xymatrix@C=15pt{
 & V_{1,2,3,4} & & & & & \\
V_{1,2,3}   \ar@{-}[ur] & V_{1,2,4}  \ar@{-}[u] & V_{1,3,4}  \ar@{-}[ul]& & & V_{1}+\langle\Delta_{2 3}, \Delta_{2 4}, \Delta_{3 4} \rangle & \\
V_{1,2}  \ar@{-}[u]  \ar@{-}[ur] & V_{1,3}  \ar@{-}[ur]  \ar@{-}[ul]& V_{1,4}  \ar@{-}[u]  \ar@{-}[ul]& & V_{1}+\langle\Delta_{2 3}\rangle\ar@{-}[ur]&  V_{1}+\langle\Delta_{2 4}\rangle \ar@{-}[u]&  V_{1}+\langle\Delta_{3 4}\rangle \ar@{-}[ul]\\
 & & & V_{1} \ar@{-}[ulll] \ar@{-}[ull] \ar@{-}[ul] \ar@{-}[ur] \ar@{-}[urr] \ar@{-}[urrr]& & & \\
}
\end{equation*}
Our generalized algorithm is only guaranteed to find those primes which are not contained in $\phi^t(F^t_* S)$ for $t \gg 0$.  However, these primes contain no information, see Remark \ref{rem.FindingBoringPrimes}, and are simply an artifact of the algorithm.
\end{example}

\bibliographystyle{skalpha}
\bibliography{CommonBib}

\def\cprime{$'$} \def\cprime{$'$}
  \def\cfudot#1{\ifmmode\setbox7\hbox{$\accent"5E#1$}\else
  \setbox7\hbox{\accent"5E#1}\penalty 10000\relax\fi\raise 1\ht7
  \hbox{\raise.1ex\hbox to 1\wd7{\hss.\hss}}\penalty 10000 \hskip-1\wd7\penalty
  10000\box7}
\providecommand{\bysame}{\leavevmode\hbox to3em{\hrulefill}\thinspace}
\providecommand{\MR}{\relax\ifhmode\unskip\space\fi MR}
\providecommand{\MRhref}[2]{%
  \href{http://www.ams.org/mathscinet-getitem?mr=#1}{#2}
}
\providecommand{\href}[2]{#2}
\begin{thebibliography}{BSTZ10}

\bibitem[BMS08]{BlickleMustataSmithDiscretenessAndRationalityOfFThresholds}
{\sc M.~Blickle, M.~Musta{\c{t}}{\u{a}}, and K.~Smith}: \emph{Discreteness and
  rationality of {F}-thresholds}, Michigan Math. J. \textbf{57} (2008), 43--61.

\bibitem[BSTZ10]{BlickleSchwedeTakagiZhang}
{\sc M.~Blickle, K.~Schwede, S.~Takagi, and W.~Zhang}: \emph{Discreteness and
  rationality of {$F$}-jumping numbers on singular varieties}, Math. Ann.
  \textbf{347} (2010), no.~4, 917--949. {\sf\scriptsize 2658149}

\bibitem[BK05]{BrionKumarFrobeniusSplitting}
{\sc M.~Brion and S.~Kumar}: \emph{Frobenius splitting methods in geometry and
  representation theory}, Progress in Mathematics, vol. 231, Birkh\"auser
  Boston Inc., Boston, MA, 2005. {\sf\scriptsize MR2107324 (2005k:14104)}

\bibitem[BS98]{BrodmannSharpLocalCohomology}
{\sc M.~P. Brodmann and R.~Y. Sharp}: \emph{Local cohomology: an algebraic
  introduction with geometric applications}, Cambridge Studies in Advanced
  Mathematics, vol.~60, Cambridge University Press, Cambridge, 1998.
  {\sf\scriptsize MR1613627 (99h:13020)}

\bibitem[CLO07]{CoxLittleOSheaIdealsVarietiesAlgorithms}
{\sc D.~Cox, J.~Little, and D.~O'Shea}: \emph{Ideals, varieties, and
  algorithms}, third ed., Undergraduate Texts in Mathematics, Springer, New
  York, 2007, An introduction to computational algebraic geometry and
  commutative algebra. {\sf\scriptsize 2290010 (2007h:13036)}

\bibitem[EH08]{EnescuHochsterTheFrobeniusStructureOfLocalCohomology}
{\sc F.~Enescu and M.~Hochster}: \emph{The {F}robenius structure of local
  cohomology}, Algebra Number Theory \textbf{2} (2008), no.~7, 721--754.
  {\sf\scriptsize MR2460693 (2009i:13009)}

\bibitem[Fed83]{FedderFPureRat}
{\sc R.~Fedder}: \emph{{$F$}-purity and rational singularity}, Trans. Amer.
  Math. Soc. \textbf{278} (1983), no.~2, 461--480. {\sf\scriptsize MR701505
  (84h:13031)}

\bibitem[GS11]{M2}
{\sc D.~R. Grayson and M.~E. Stillman}: \emph{Macaulay2, a software system for
  research in algebraic geometry}, 2011.

\bibitem[GP08]{GreuelPfisterASingularIntroduction}
{\sc G.-M. Greuel and G.~Pfister}: \emph{A {\bf {s}ingular} introduction to
  commutative algebra}, extended ed., Springer, Berlin, 2008, With
  contributions by Olaf Bachmann, Christoph Lossen and Hans Sch{\"o}nemann,
  With 1 CD-ROM (Windows, Macintosh and UNIX). {\sf\scriptsize 2363237
  (2008j:13001)}

\bibitem[Kat08]{KatzmanParameterTestIdealOfCMRings}
{\sc M.~Katzman}: \emph{Parameter-test-ideals of {C}ohen-{M}acaulay rings},
  Compos. Math. \textbf{144} (2008), no.~4, 933--948. {\sf\scriptsize MR2441251
  (2009d:13030)}

\bibitem[KS11]{KatzmanSchwedeM2FSplitting}
{\sc M.~Katzman and K.~Schwede}: \emph{Fsplitting}, A Macaulay2 package
  implementing an algorithm for computing compatibly Frobenius split
  subvarieties, freely available from
  \url{http://katzman.staff.shef.ac.uk/FSplitting/}, 2011.

\bibitem[Knu09]{KnutsonFrobeniusSplittingPointCountingAndDegeneration}
{\sc A.~Knutson}: \emph{Frobenius splitting, point counting and degeneration},
  arXiv:0911.4941v1.

\bibitem[KLS10]{KnutsonLamSpeyerProjectionsOfRichardson}
{\sc A.~Knutson, T.~Lam, and D.~E. Speyer}: \emph{Projections of {R}ichardson
  varieties}, arXiv:1008.3939.

\bibitem[KM09]{KumarMehtaFiniteness}
{\sc S.~Kumar and V.~B. Mehta}: \emph{Finiteness of the number of compatibly
  split subvarieties}, Int. Math. Res. Not. IMRN (2009), no.~19, 3595--3597.
  {\sf\scriptsize 2539185 (2010j:13012)}

\bibitem[Sch09]{SchwedeFAdjunction}
{\sc K.~Schwede}: \emph{{$F$}-adjunction}, Algebra Number Theory \textbf{3}
  (2009), no.~8, 907--950.

\bibitem[Sch10]{SchwedeCentersOfFPurity}
{\sc K.~Schwede}: \emph{Centers of {$F$}-purity}, Math. Z. \textbf{265} (2010),
  no.~3, 687--714. {\sf\scriptsize 2644316 (2011e:13011)}

\bibitem[Sha07]{SharpGradedAnnihilatorsOfModulesOverTheFrobeniusSkewPolynomialRing}
{\sc R.~Y. Sharp}: \emph{Graded annihilators of modules over the {F}robenius
  skew polynomial ring, and tight closure}, Trans. Amer. Math. Soc.
  \textbf{359} (2007), no.~9, 4237--4258 (electronic). {\sf\scriptsize
  MR2309183 (2008b:13006)}

\end{thebibliography}

\end{document}